\documentclass[10pt, BCOR=2mm]{amsart}
\usepackage{amssymb}
\usepackage{amscd} 
\usepackage{tikz-cd}
\usepackage{multicol,caption}
\usepackage[
  margin=3.6cm,
  includefoot,
  footskip=20pt,
]{geometry}
\usepackage{mathscinet}
\usepackage[nobysame,alphabetic,initials,msc-links,lite,backrefs]{amsrefs}
\usepackage{mathtools, mathrsfs, color} 
\usepackage{graphicx}
\usepackage{tikz,tikz-cd}
\usepackage[all,cmtip]{xy}
\numberwithin{equation}{section}

\def\today{\number\day\space\ifcase\month\or   January\or February\or
   March\or April\or May\or June\or   July\or August\or September\or
   October\or November\or December\fi\   \number\year}

\numberwithin{equation}{section}
\newtheorem{theorem}{Theorem}[section]
\newtheorem{lemma}[theorem]{Lemma}
\newtheorem{proposition}[theorem]{Proposition}
\newtheorem{corollary}[theorem]{Corollary}
\newtheorem{definition}[theorem]{Definition}
\newtheorem{remark}[theorem]{Remark}

\newtheorem{example}[theorem]{Example}
\newcommand{\Z}{{\mathbb{Z}}}
\newcommand{\R}{{\mathbb{R}}}

\newcommand{\N}{\mathbb{N}}

\newcommand{\pf}{{\operatorname{pf}}}
\newcommand{\deta}{{\operatorname{det}}}
\newcommand{\K}{\mathrm{K}} % $\K$-theory
\newcommand{\GL}{\mathrm{GL}}

\newcommand{\te}{\theta}

\newcommand{\SL}{\mathrm{SL}}
\newcommand{\Sp}{\mathrm{Sp}}
\newcommand{\Mp}{\mathrm{Mp}}

\DeclareMathOperator{\Aut}{Aut}
\DeclareMathOperator{\Tr}{Tr}

\DeclareMathOperator{\id}{id}

\newcommand{\bn}{\noindent \begin{nummer} \rm}
\newcommand{\en}{\end{nummer}}

\title{Tracing projective modules \\ over noncommutative orbifolds}   % ???

\author[S.~Chakraborty]{Sayan Chakraborty} 
\address{Stat-Math unit, Indian Statistical Institute, 203 Barrackpore Trunk Road, Kolkata 700 108, India.}
\email{sayan2008@gmail.com}

\begin{document}

\begin{abstract}For an action of a finite cyclic group $F$ on an $n$-dimensional noncommutative torus $A_\theta,$ we give sufficient conditions when the fundamental projective modules over $A_\theta$, which determine the range of the canonical trace on $A_\theta,$ extend to projective modules over the crossed product C*-algebra $A_\theta \rtimes F.$  Our results allow us to understand the range of the canonical trace on $A_\theta \rtimes F$, and determine it completely for several examples including the crossed products of 2-dimensional noncommutative tori with finite cyclic groups and the flip action of $\Z_2$ on any $n$-dimensional noncommutative torus. As an application, for the flip action of $\Z_2$ on a simple $n$-dimensional torus $A_\theta$, we determine the Morita equivalence class of $A_\theta \rtimes \Z_2,$ in terms of the Morita equivalence class of $A_\theta.$
\end{abstract}

\maketitle \pagestyle{myheadings} \markboth{S.~Chakraborty}{Tracing projective modules over noncommutative orbifolds}

\section*{Introduction}\label{sec:intro}

For $n \geq 2$, let $\mathcal{T}_{n}$ denote the space of all $n \times n$ real skew-symmetric  matrices. The $n$-dimensional noncommutative torus $A_{\theta}$ is
the universal C*-algebra  generated by unitaries $U_1$, $U_2$,  $U_3$, $\cdots$, $U_n$
subject to the relations
\begin{equation}\label{eq:ccr}
	U_k U_j = e^ {2 \pi i \theta_{jk} } U_j U_k
\end{equation}
for $j, k = 1, 2, 3, \cdots, n$, where $\theta:=(\theta_{jk}) \in \mathcal{T}_{n}$. For the 2-dimensional noncommutative tori, since $\theta$ is determined by only one real number, $\theta_{12},$ we will denote $\theta_{12}$ by $\theta$ again and the corresponding 2-dimensional noncommutative torus by $A_{\theta}$.

There is a canonical action of $\SL(2,\Z)$ on two dimensional noncommutative tori, which is given by sending $U_1$ to $e^ {\pi i a c\theta_{12}}U_1^aU_2^c$ and $U_2$ to $e^ {\pi i b d\theta_{12}}U_1^bU_2^d,$ for a matrix $\begin{pmatrix}
  a & b \\
  c & d
 \end{pmatrix}\in \SL(2,\Z).$ This action was further generalised to the higher dimensional noncommutative tori. It was pointed out in \cite{JL15} that the right replacement of the group $\SL(2,\Z)$ is  
 $$\Sp(n, \Z, \theta):= \{ W \in \GL(n,\Z) : W^T\theta W = \theta\}.$$ 
 Then there is a natural action of $\Sp(n, \Z, \theta)$ on the $n$-dimensional noncommutative torus $A_{\theta}.$ It is easy to see that $\Sp(2, \Z, \theta)$ is exactly $\SL(2,\Z).$ 
 
 The study of crossed product C*-algebras associated to finite group actions on noncommutative tori goes back to the work of Bratteli, Elliott, Evans and  Kishimoto (\cite{BEEK91}). However they only looked at the action of $\Z_2$ on the C*-algebra $A_{\theta},$ for 2-dimensional tori. Recall that the action of $\Z_2$ on any $n$-dimensional $A_{\theta},$ often called the \emph{flip action}, is defined by sending $U_i$ to $U_i^{-1}.$  Note that the above action is basically given by the matrix $-\id_n \in \Sp(n, \Z, \theta),$ where $\id_n$ is the $n \times n$ unit matrix. Later various other authors studied actions of other finite cyclic subgroups of $\SL(2,\Z)$ on 2-dimensional noncommutative tori, see \cite{BW07}, \cite{ELPW10}, \cite{Wal95}, \cite{Wal00}.  Motivated by the 2-dimensional results, it is also natural to consider a finite cyclic group $F$ inside $\Sp(n, \Z, \theta)$ and consider the crossed product $A_{\theta}\rtimes F,$ for an $n$-dimensional torus $A_\theta.$ We may call such a crossed product a \emph {noncommutative orbifold.}

 The authors in \cite{ELPW10} and \cite{BCHL18} considered actions of cyclic subgroups of $\SL(2,\Z)$ on 2-dimensional noncommutative tori. Along with K-theory computations of the corresponding crossed products, the authors computed the images of the canonical tracial states of such algebras. The recent development of the classification program of C*-algebras allowed them to  deduce results about isomorphism and Morita equivalence classes of such algebras, when the algebras are simple. One of the major facts they used is that the algebras are simple AH algebras when $\theta$ is irrational (for the finite group actions, the algebras are even AF). Then the algebras are classifiable in the sense of Elliot's classification program.

 In \cite{JL15}, Jeong and Lee, and in \cite{He19}, He studied actions of finite subgroups of $\Sp(n, \Z, \theta)$ on an $n$-dimensional $A_\theta,$ and found many of such crossed products to be classifiable, when $\theta$ is non-degenerate (see Definition~\ref{def:nondegenerate}) so that $A_\theta$ is simple. However they did not discuss  isomorphism and Morita equivalence classes of the crossed products.  Our paper is a first attempt towards these kind of results for the higher dimensional cases.    
 
 To understand isomorphism and Morita equivalence classes of such noncommutative orbifolds it is necessary to compute the K-theory of the orbifolds and understand the ranges of the canonical tracial states of the algebras. While the dimensions of the K-groups are known (from \cite{LL12}), the tracial ranges are not understood. Our main results help to understand which numbers belong to the tracial ranges of the orbifolds, and even determine the tracial ranges completely for several examples.
 
 To understand the tracial range of an orbifold, one should first understand the same for the noncommutative torus itself. This was done by Elliott in \cite{Ell84}. To give an overview of our results, we recall the tracial range result from \cite{Ell84}.
 For an integer $p$ with $1 \leq p \leq \frac{n}{2}$, if we denote the sub-matrix $M^{\theta}_{I}$ of $\theta$ consisting of rows and columns indexed by the numbers $i_1, i_2, ..., i_{2p}$ for some $i_1 < i_2 < ... < i_{2p}$, $I:=\left(i_{1}, i_{2}, \ldots, i_{2 p}\right)$, then Elliott's result may be stated as 
\begin{equation*}
	\operatorname{Tr}\left(\K_{0}\left(A_{\theta}\right)\right)=\mathbb{Z}+\sum_{0<|I| \leq n} \pf(M_{I}^{\theta}) \mathbb{Z},
\end{equation*}
where $|I|:=2m$ for $I=\left(i_{1}, i_{2}, \ldots, i_{2 m}\right)$ and $\pf$ denotes the pfaffian. Here $\Tr$ denotes the canonical tracial state on $A_\theta.$

It was observed in \cite{Cha20} that for each such $I$, there is a projective module $\mathcal{E}_{I}^{\theta}$ over $A_\theta,$ trace of which is exactly $\pf(M_{I}^{\theta}),$ assuming $\pf(M_{I}^{\theta})\neq 0.$ This module is governed by an element $g_{I, \Sigma} \in \mathrm{SO}(n, n|\mathbb{Z}).$ 
 Here $\mathrm{SO}(n, n|\mathbb{Z})$ is a certain subgroup of the group of linear transformations of the space $\mathbb{R}^{2 n}$ preserving the quadratic form $x_{1} x_{n+1}+x_{2} x_{n+2}+\cdots+x_{n} x_{2 n}$ (see Section~\ref{sec:k_theoer_nt} for more details). The modules of such kind are called \emph{fundamental projective modules}.

Now coming back to the crossed products of $n$-dimensional tori $A_\theta$ with a finite cyclic group $F \subset \Sp(n, \Z, \theta),$ if $\Tr^F$ denotes the canonical trace on $A_{\theta} \rtimes F,$  the regular representation $A_{\theta} \rtimes F \hookrightarrow \mathrm{M}_N(A_{\theta})$ gives 
 $$
	\Tr^F(\K_0(A_{\theta} \rtimes F))\subseteq \frac{1}{N}\Tr(\K_0(A_{\theta}))=\frac{1}{N}\Bigg( \mathbb{Z}+\sum_{0<|I| \leq n} \pf(M_{I}^{\theta}) \mathbb{Z}\Bigg).
$$
 Our main theorem (Theorem~\ref{intro:th1}) determines when the term $\frac{1}{N}\pf(M_{I}^{\theta})$ lies in the left hand side of the above equation. The proof of the theorem involves extending the modules $\mathcal{E}_{I}^{\theta}$ to modules over the crossed products using so-called \emph {metaplectic operators}, which were already used by the author (in a joint work with Luef) in \cite{CL19} to extend a specific type of modules (\emph{Bott classes}) to modules over the crossed products.
 
 Let $\mathcal{R}$ denote the subgroup of $\mathrm{SO}(n, n|\mathbb{Z})$ generated by the elements of $\GL(n,\Z).$
\begin{theorem}\label{intro:th1} (Theorem~\ref{thm:mainW}, Theorem~\ref{th:main_trace})
 With all the notations introduced above, assume $\pf(M_{I}^{\theta}) \neq 0.$ Let $W\in \GL(n,\Z)$ be of finite order such that $W^t\theta W = \theta$ and $F:=\langle W \rangle.$  Suppose $g_{I, \Sigma} F(g_{I, \Sigma})^{-1} \subset \mathcal{R}$ inside $\mathrm{SO}(n, n|\mathbb{Z}).$ Then  $\mathcal{E}_{I}^{\theta}$ becomes a finitely generated, projective module over $A_{\theta} \rtimes F$ and $\frac{1}{N}\pf(M_{I}^{\theta}) \in \Tr^F(\K_0(A_{\theta} \rtimes F)),$ where $N$ is the order of $W$. 
	\end{theorem}

The condition in the above theorem is easy to check for many examples. In fact, we provide some examples with explicit tracial range computations. These examples include the two dimensional cases and the flip action of $\Z_2$. It is worthwhile to explicitly state the consequences for the flip action here in the introduction, since the results were unknown to the author. For the tracial range we get,
	$$\Tr^{\Z_2}(\K_0(A_\theta\rtimes \Z_2)) = \frac{1}{2}\Tr(\K_0(A_\theta)),$$
	for any $\theta$ in  $\mathcal{T}_n.$ And as a corollary we have, 
	
	\begin{corollary}\label{intro:cor1}(Corollary~\ref{cor:Z_2_main})
		Let $\theta_{1}, \theta_{2} \in \mathcal{T}_n $ be non-degenerate.
		Let $\Z_2$ act on $A_{\theta_{1}}$ and $A_{\theta_{2}}$ by the flip actions.  Then $A_{\theta_{1}}\rtimes \Z_2$ is strongly Morita equivalent to $A_{\theta_{2}}\rtimes \Z_2$ if and only if $A_{\theta_{1}}$ is strongly Morita equivalent to $A_{\theta_{2}}.$
 	 \end{corollary}
 	 It is worth mentioning that the only action of a finite cyclic subgroup of $\Sp(3, \Z, \theta)$ on a 3-dimensional torus $A_\theta,$ when $\theta$ is non-degenerate, is the flip action (\cite[Theorem 1.4]{JL15}).  
 	 
 	Apart from the applications in classification of C*-algebras, the computations of the ranges of tracial states turn out to be useful in physics (Bellisard's gap labelling theorem, in particular). Our results are similar to results which appeared in connection with the study of a twisted version of the gap labelling theorem, recently conjectured in \cite{BM18}. We hope that our techniques will be helpful for a better understanding of the conjecture.  
	
	This article is organised as follows: in Section~\ref{sec:twisted_C*_algebras} we recall the definition of twisted group C*-algebras and give relevant examples. In Section~\ref{sec:k_theoer_nt}, we discuss the fundamental projective modules over noncommutative tori. Section~\ref{sec:orbifolds_projective} deals with  extending the fundamental modules to modules over orbifolds, and proving Theorem~\ref{intro:th1} (Theorem~\ref{thm:mainW}). In the last section, Section~\ref{Sec:ntflip}, the proof of Theorem~\ref{intro:th1} (Theorem~\ref{th:main_trace}) about the ranges of the  canonical traces on orbifolds is discussed along with various examples. We also discuss the results about Morita equivalence classes of orbifolds, along with Corollary~\ref{intro:cor1} in Section~\ref{Sec:ntflip}.

	\textbf {Notation}: $e(x)$ will always denote the number $e^{2\pi i x}$, and $\id_m$ will be the $m\times m$ unit matrix.

	%%%%%%%%%%%%%%%%%%%%%%%%%%%%%%%%%%%%%%%%%%%%%%%%%%%%%%%%%%%%%%%%%%%%%%%%%%%%%%%%%%%%%%%%%%%%%%%%%%%%%%%%%%%%%%%%%%%%%%%%%%%%%%%%%%%%%%%%%%%%%%%%%%%%%%%%%%%%%%%%%%%%%%%%%%%%%%%%%%%%%%%%%%%%%%%%%%%%%%%%%%%%%%%%%%%%%%%%%%%%%%%%%%%%%%%%%%%%%%%%%%%%%%%%
	
	\section{Twisted group C*-algebras and noncommutative orbifolds}\label{sec:twisted_C*_algebras}
	Let $G$ be a discrete group.  A map $\omega: G\times G \to \mathbb{T}$ is called a \emph{2-cocycle} if 
$$ \omega(x, y) \omega(xy, z)= \omega(x, yz) \omega(y, z)
\label{cocycle}
$$
whenever $x, y, z \in G $, and if 
$$ \omega(x, 1) = 1 = \omega(1, x)\label{units} $$
for any $x \in G$.

The \emph{$\omega$-twisted left regular representation} of the group $G$ is given by the formula:

$$(L_{\omega}(x)f)(y) = \omega(x, x^{-1}y)f(x^{-1}y),$$ for $f \in l^2(G)$. The \emph{reduced twisted group C*-algebra} $C^*(G,\omega)$  is defined as the sub-C*-algebra of $B(l^2(G))$ generated by the $\omega$-twisted left regular representation of the group $G$. Since we do not talk about full group C*-algebras in this paper, we simply call $C^*(G,\omega)$ the twisted group C*-algebra of $G$ with respect to $\omega.$ When $\omega=1,$ $C^*(G,\omega)=:C^*(G)$ is the usual reduced group C*-algebra of $G.$  We refer to \cite[Section 1]{ELPW10} for more on twisted group C*-algebras and the details of the above construction.

\begin{example}\label{ex:nt}
	\normalfont
Let $G$ be the group $\Z^n$. For each $\theta \in \mathcal{T}_{n}$, construct a 2-cocycle on $G$ by defining $\omega_\theta(x, y) = e(\langle -\theta x,y \rangle )$. The corresponding twisted group C*-algebra $C^*(G, \omega_\theta)$ is isomorphic to the $n$-dimensional noncommutative torus $A_\te$, which was defined in the introduction. 

  \end{example}
\begin{example}\label{ex:orbifold}
\normalfont
Suppose $W$ be an invertible $n \times n$ matrix of finite order with integer entries. Let $F := \langle W \rangle$ act on $\Z^n$ by usual matrix multiplication with vectors. Let us also take $\theta \in \mathcal{T}_{n}$. We assume in addition  that $W$ is a \emph{$\theta$-symplectic} matrix, i.e. $W^t\theta W = \theta$. Then we can define a 2-cocycle $\omega_\theta '$ on $G:=\Z^n \rtimes F$ by $\omega_\theta '((x,s),(y,t)) = \omega_\theta (x,s\cdot y)$.  Sometimes one calls the corresponding twisted group C*-algebra, $C^*(G, \omega_\theta')$, a \emph{noncommutative orbifold}. We will come back to this example in Section~\ref{sec:orbifolds_projective}.
 \end{example}

\section{K-theory generators of noncommutative tori}\label{sec:k_theoer_nt}

\subsection{Projective modules over noncommutative tori}\label{sec:proj_mod_nt}

 In \cite{RS99}, Rieffel and Schwarz defined (densely) an action of the group $\mathrm{SO}(n, n|\Z)$ on $\mathcal{T}_{n}$. Recall that $\mathrm{SO}(n, n|\Z)$ is the subgroup of $\mathrm{GL}(2n, \R),$ which contains matrices, with integer entries and of determinant 1,  of the following $2 \times 2$ block form: $$
g=\left(\begin{array}{ll}
A & B \\
C & D
\end{array}\right),
$$
where $A, B, C$ and $D$ are arbitrary $n \times n$ matrices over $\Z$ satisfying
$$
A^{t} C+C^{t} A=0, \quad B^{t} D+D^{t} B=0 \quad \text { and } \quad A^{t} D+C^{t} B=\id_n.
$$
The action of $\mathrm{SO}(n, n|\Z)$ on $\mathcal{T}_{n}$ is defined as
$$
g \theta:=(A \theta+B)(C \theta+D)^{-1}
$$
whenever $C \theta+D$ is invertible. The subset of $\mathcal{T}_{n}$ on which the action of every $g \in \mathrm{SO}(n, n|\Z)$ is defined, is dense in $\mathcal{T}_{n}$ (see \cite[page 291]{RS99}). We have the following theorem due to Hanfeng Li. 
\begin{theorem}\label{thm:li_Morita}(\cite[Theorem 1.1]{Li04})
	  For any $\theta \in \mathcal{T}_{n}$ and $g \in \mathrm{SO}(n, n|\Z),$ if $g \theta$ is defined then $A_{\theta}$ and $A_{g \theta}$ are strongly Morita equivalent.
\end{theorem}

For any $R \in \mathrm{GL}(n,\Z),$ let us denote by $\rho(R)$ the matrix $\left(\begin{array}{cc}R & 0 \\ 0 & \left(R^{-1}\right)^{\mathrm{t}}\end{array}\right) \in \mathrm{SO}(n, n|\Z) ,$ and for any $N \in \mathcal{T}_{n} \cap \mathrm{M}_{n}(\Z),$ we denote by $\mu(N)$ the matrix $\left(\begin{array}{cc}\id_n & N \\ 0 & \id_n\end{array}\right) \in \mathrm{SO}(n, n|\Z).$
Notice that the noncommutative tori corresponding to the matrices $\rho(R) \theta=R \theta R^{t}$ and $\mu(N) \theta=\theta+N$ are both isomorphic to $A_{\theta}$. Also define

$$\mathrm{SO}(n, n|\Z) \ni \sigma_{2 p}:=\left(\begin{array}{cccc}0 & 0 & \mathrm{id}_{2 p} & 0 \\ 0 & \mathrm{id}_{n-2 p} & 0 & 0 \\ \mathrm{id}_{2 p} & 0 & 0 & 0 \\ 0 & 0 & 0 & \mathrm{id}_{n-2 p}\end{array}\right), 1 \leqslant p \leqslant n / 2.$$
 
 We recall the approach of Rieffel \cite{Rie88} to find the $A_{\sigma_{2 p}\theta}-A_\theta$ bimodule and follow the presentation in \cite{Li04}.

We fix $1 \leqslant p \leqslant n / 2,$ and let $q \in \N$ such that  $n=2p + q$. Let us write $\theta \in \mathcal{T}_n$ as $\left( \begin{array}{ccc}
\theta_{11} & \theta_{12}\\
  \theta_{21}  & \theta_{22}\
  \end{array} \right),$ partitioned into four sub-matrices $\theta_{11},\theta_{12},\theta_{21},\theta_{22},$ and assume $\theta_{11}$ to be an invertible $2p\times 2p$ matrix.   Define a new cocycle $\omega_{\theta'}$  on $\Z^n$ by $\omega_{\theta'}(x, y)  = e(\langle -\theta'x , y \rangle/2 )$, where $$
\theta'    =  
\left( \begin{array}{cc}
\theta_{11}^{-1}  &  -\theta_{11}^{-1}\theta_{12}\\
\theta_{21}\theta_{11}^{-1}  &  
\theta_{22} - \theta_{21}\theta_{11}^{-1}\theta_{12}
\end{array} \right) = \sigma_{2 p}\theta.       
$$

Set $\mathcal{A} = C^* ( \Z^n, \omega_{\theta} )$ and
$\mathcal{B} = C^* ( \Z^n, \omega_{\theta'} )$.  Let $M$ be the group $\R^p \times \Z^q$, $G := M \times \widehat{M}$ and $\langle \cdot , \cdot \rangle$ be the natural pairing between  $M$ and its dual group $\widehat{M}$ (our notation does not distinguish between the pairing of a group and its dual group, and the standard inner product on a linear space). 
Consider the Schwartz space $\mathcal{E}^\infty := \mathscr{S}(M)$
consisting of smooth and rapidly decreasing complex-valued functions on $M$.

Denote by $\mathcal{A}^\infty = \mathcal{S} (\Z^n, \omega_{\theta})$ and  $\mathcal{B}^\infty = \mathcal{S} (\Z^n, \omega_{\theta'})$ 
the dense sub-algebras of $\mathcal{A} $ and $\mathcal{B} $, respectively, consisting of formal series (of the variables $\{U_i\}$) with rapidly decaying coefficients.
Let us consider the following $(2p+2q) \times (2p + q)$ real valued matrix:
\begin{equation}\label{eq:Tmap}
	T = \left( \begin{array}{cc}
T_{11} & 0\\
0  & \id_q\\
T_{31}  & T_{32}
\end{array} \right), 
\end{equation}
where $T_{11}$ is an invertible matrix such that $T_{11}^tJ_0T_{11} = \theta_{11}$, $J_0
  := \left( \begin{array}{ccc}
0 & \id_p\\
  -\id_p  & 0\\
  \end{array} \right)$, $T_{31} = \theta_{21}$ and $T_{32}$ is any $q \times q$ matrix such that $\theta_{22} =  T_{32}-T_{32}^t.$ For our purposes, we take $T_{32}=\theta_{22}/2.$
  
 We also define  the following $(2p+2q) \times (2p + q)$ real valued matrix:
 $$
S = \left( \begin{array}{cc}
J_0(T_{11}^t)^{-1} & -J_0(T_{11}^t)^{-1}T_{31}^t\\
0  & \id_q\\
0  & T^t_{32}
\end{array} \right).
$$ 
Let 
$$
 J = \left( \begin{array}{ccc}
J_0 & 0 & 0\\
0 & 0 & \id_q\\
0 & -\id_q & 0\\
\end{array} \right) 
$$ and $J'$ be the matrix obtained from $J$ by replacing the negative entries of it by zeroes. Note that $T$ and $S$ can be thought as maps from $(\R^{n})^*$ to $\mathbb{R}^{p} \times (\mathbb{R}^ p)^* \times \mathbb{R}^{q} \times (\mathbb{R}^ q)^*$ (see  the definition of an embedding map in \cite[Definition 2.1]{Li04}), and  $S\left(\mathbb{Z}^{n}\right), T\left(\mathbb{Z}^{n}\right) \subseteq \mathbb{R}^{p} \times (\mathbb{R}^ p)^* \times \mathbb{Z}^{q} \times (\mathbb{R}^ q)^*$. Then we can think of $S\left(\mathbb{Z}^{n}\right), T\left(\mathbb{Z}^{n}\right)$ as in $G$
via composing $\left.S\right|_{\mathbb{Z}^{n}}, \left.T\right|_{\mathbb{Z}^{n}}$ with the natural covering map $\mathbb{R}^{p} \times (\mathbb{R}^ p)^* \times \mathbb{Z}^{q} \times (\mathbb{R}^ q)^* \rightarrow G$. 
Let $P'$ and $P''$ be the canonical projections of $G$ to $M$ and $\widehat{M}$, respectively, and let
$$T':=P'\circ T, \quad T'':=P''\circ T, \quad  S':= P'\circ S, \quad S'':= P''\circ S.$$  Then the following formulas define a $\mathcal{B}^\infty$-$\mathcal{A}^\infty$ bimodule structure on $\mathcal{E}^\infty$:

\begin{equation} \label{eq:proj_mod_T}
(f U_l^\theta)(x)= e(\langle -T(l), J^{\prime} T(l)/2\rangle)\langle x, T^{\prime \prime}(l) \rangle f(x-T^{\prime}(l)), 
\end{equation} 

\begin{equation} \label{eq:proj_mod_inner}
\langle f,g \rangle_{\mathcal{A}^\infty}(l)= e(\langle-T(l), J^{\prime} T(l)/2\rangle)\int_{G} \langle x,-T^{\prime \prime}(l) \rangle g(x+T^{\prime}(l))\bar{f}(x) dx, 
\end{equation}

\begin{equation} \label{eq:proj_mod_Tprime}
(U_l^{\sigma_{2 p}\theta}  f)(x)= e(\langle-S(l), J^{\prime} S(l)/2\rangle)\langle x, -S^{\prime \prime}(l) \rangle f(x+S^{\prime}(l)), 
\end{equation} 

\begin{equation} \label{eq:proj_mod_inner_Tprime}
{}_{\mathcal{B}^\infty}\langle f,g \rangle(l)= e(\langle S(l), J^{\prime} S(l)/2\rangle)\int_{G} \langle x,S^{\prime \prime}(l) \rangle \bar{g}(x+S^{\prime}(l))f(x) dx, 
\end{equation}
 where $U_l^\theta, U_l^{\sigma_{2 p}\theta}$ denote the canonical unitaries with respect to the group element $l \in \Z^n$ in $\mathcal{A}^\infty $ and $\mathcal{B}^\infty$, respectively. See Proposition 2.2 in \cite{Li04} for the following well-known result.

\begin{theorem}[Rieffel]\label{theorem:rieffel_proj_mod}
The smooth module $\mathcal{E}^\infty$, with the above structures, is an $\mathcal{B}^\infty$-$\mathcal{A}^\infty$  Morita equivalence bimodule  which can be extended to a strong Morita equivalence between $\mathcal{B}$ and $\mathcal{A}$. 

\end{theorem}
Let $\mathcal{E}$ denote the completion of $\mathcal{E}^\infty$ with respect to the  $C^*$-valued inner products given above. Now  $\mathcal{E}$ 
becomes a right  projective $A$-module which is also finitely generated (see the discussion preceding Proposition 4.6 of \cite{ELPW10}). Note that $\mathcal{E}$ is a Morita equivalence bimodule between $\mathcal{B}=A_{\sigma_{2 p}\theta}$ and $\mathcal{A}=A_\theta.$

\subsection{Fundamental projective modules}

For a definition of the pfaffian of a skew-symmetric matrix $A$, $\pf(A),$ we refer to \cite[Definition 3.1]{Cha20}. We start with the following remark. 
\begin{remark}\label{rmk:trace_of_heisenberg}
\normalfont
	The trace of the module $\mathcal{E}$, which was computed by Rieffel \cite{Rie88}, is exactly the absolute value of the pfaffian of the upper left $2p \times 2p$ corner of the matrix $\theta,$ which is $\theta_{11}$. Indeed, as \cite[Proposition 4.3, page 289]{Rie88} says that trace of $\mathcal{E}$ is $|\deta \widetilde{T}|$, where 
$$
\widetilde{T}= \left( \begin{array}{cc}
T_{11} & 0\\
0  & \id_q\\
\end{array} \right),
$$ the relation $T_{11}^tJ_0T_{11} = \theta_{11}$ and the fact $\deta (J_0) =1$ give the claim.
\end{remark}

%%%%%%%%%%%%%%%%%%%%

%%%%%%%%%%%%%%%%%%%%
Let $p$ be an integer such that $1 \leq p \leq \frac{n}{2}$.
\begin{definition} 
\normalfont
A  \emph{$2p$-pfaffian minor} (or just pfaffian minor)  of a skew-symmetric matrix $A$ is the pfaffian of a sub-matrix $M^A_{I}$ of $A$ consisting of rows and columns indexed by $i_1, i_2, ..., i_{2p}$ for some numbers $i_1 < i_2 < ... < i_{2p}$ and $I:=\left(i_{1}, i_{2}, \ldots, i_{2 p}\right)$. 	
\end{definition}
\noindent Note that the number of  $2p$-pfaffian minors is  ${n \choose 2p}$ and the number of all pfaffian minors is $2^{n-1}-1$.

%%%%%%%%%%%
 Let $\theta \in \mathcal{T}_n.$ We will now see that for each non-zero pfaffian minor of $\theta,$ we can construct a projective module over $A_\theta$ such that the trace of this module is exactly the pfaffian minor.  Fix $1 \leq p \leq \frac{n}{2}$. Choose $I:=\left(i_{1}, i_{2}, \ldots, i_{2 p}\right)$ for $i_1 < i_2 < ... < i_{2p}$, and assume the pfaffian minor $\pf(M^{\theta}_I)$ is non-zero (so that $M^{\theta}_I$ is invertible). Choose a permutation $\Sigma \in \mathcal{S}_{n}$ such that $\Sigma(1) = i_1,  \Sigma(2) = i_2,\cdots, \Sigma(2p) = i_{2p}.$ If $U_1, U_2, \cdots, U_n$ are generators of $A_\theta$, there exists an $n \times n$ skew-symmetric matrix, denoted by  $\Sigma(\theta)$, such that $U_{\Sigma(1)},  U_{\Sigma(2)}, \cdots, U_{\Sigma(n)}$ are generators of $A_{\Sigma(\theta)}$ and $A_{\Sigma(\theta)} \cong A_\theta.$ Note that the upper left $2p \times 2p$ block $\Sigma(\theta)$ is exactly $M_{I}^{\theta},$ which is invertible. Now consider the projective module constructed as completion of $\mathscr{S}(\R^p \times \Z^{n-2p})$ over $A_{\Sigma(\theta)}$ as in the previous subsection and denote it by  $\mathcal{E}_{I}^{\theta}$. The trace of this module is the pfaffian of $M_{I}^{\theta}$ by the remark above, which is $\sum_{\xi \in \Pi}(-1)^{|\xi|}\prod^{p}_{s=1}\theta_{i_{\xi(2s-1)}i_{\xi(2s)}}.$  Varying $p$, and assuming that all the pfaffian minors are non-zero, we get $2^{n-1}-1$ projective modules. We call these $2^{n-1}-1$ elements the \emph{fundamental projective modules}. 

We recall the following fact due to Elliott which will play a key role. 
\begin{theorem}[Elliott]\label{elliott_image_of_trace}
Let $\theta$ be a skew-symmetric real $n\times n$ matrix. Then $\Tr(\K_0(A_\theta))$  is the range of the exterior exponential
 $$\operatorname{exp}(\theta):\Lambda^{\operatorname{even}}\Z^n\rightarrow \R.$$
\end{theorem}
 \noindent We refer to (\cite[Theorem 3.1]{Ell84}) for the definition of exterior exponential and the proof of the above theorem. The range of the exterior exponential is well known and is given below as a corollary of the above theorem:
 \begin{corollary}\label{imageoftrace}
	\sloppy $\Tr(\K_0(A_{\theta}))$ is the subgroup of $\R$
generated by $1$ and the numbers $\sum_{\xi}(-1)^{|\xi|}\prod^{m}_{s=1}\theta_{j_{\xi(2s-1)}j_{\xi(2s)}}$ for
$1\le j_1<j_2< {\cdots}<j_{2m}\le n$, where the sum is taken over all elements $\xi$ of
the permutation group $\mathcal{S}_{2m}$ such that $\xi(2s-1)<\xi(2s)$ for all $1\le s\le m$
and $\xi(1)<\xi(3)<\cdots<\xi(2m-1)$.\end{corollary}

\sloppy Noting that $\sum_{\xi}(-1)^{|\xi|}\prod^{m}_{s=1}\theta_{j_{\xi(2s-1)}j_{\xi(2s)}}$ is exactly the pfaffian of $M_{I}^{\theta},$ where $I=\left(i_{1}, i_{2}, \ldots, i_{2 m}\right),$ we have

\begin{equation}\label{eq:image_of_trace_nt}
	\operatorname{Tr}\left(\K_{0}\left(A_{\theta}\right)\right)=\mathbb{Z}+\sum_{0<|I| \leq n} \pf(M_{I}^{\theta}) \mathbb{Z},
\end{equation}
where $|I|:=2m$ for $I=\left(i_{1}, i_{2}, \ldots, i_{2 m}\right).$

So for a non-zero $\pf(M_{I}^{\theta})$, $I=\left(i_{1}, i_{2}, \ldots, i_{2 p}\right),$ we have constructed a projective module $\mathcal{E}_{I}^{\theta}$ over $A_\theta,$ whose trace is $\pf(M_{I}^{\theta}).$ A quick thought shows that $\mathcal{E}_{I}^{\theta}$ is an equivalence bimodule between $A_\theta$ and $A_{g_{I, \Sigma}\theta}$ for some $g_{I, \Sigma} \in \mathrm{SO}(n, n|\Z).$ Indeed, let $R_{I}^{\Sigma}$ be the permutation matrix corresponding to the permutation $\Sigma$. Note that $\Sigma(\theta) = \rho\left(R_{I}^{\Sigma}\right)\theta.$    Then clearly $g_{I, \Sigma}=\sigma_{2 p}  \rho\left(R_{I}^{\Sigma}\right).$

\subsection{Explicit generators of $\K_0(A_\theta)$ for a general $\theta \in \mathcal{T}_n$}\label{sec:generators}

 Consider the matrix $Z \in \mathcal{T}_n$ whose entries above the diagonal are all 1: 
  $$Z = 
   \left( \begin{array}{cccccccc}
0 & 1 & \cdots  &  &  & \cdots &1  \\
-1 & \ddots &\ddots  & &   &  &\vdots \\
\vdots & \ddots  &   &  & &  &\\
 &  & &  &  &  &\\
&  &   & &  & \ddots  & \vdots\\
  \vdots &  &  & &   \ddots  & \ddots & 1\\
-1 & \cdots &  &  &  \cdots &  -1 & 0\\
\end{array} \right) .
$$Now, for any $\theta \in \mathcal{T}_n$, there exists some positive integer $t$, such that all the pfaffian minors of $\mu(tZ)\theta=\theta + tZ$ are positive (see \cite[Proposition 4.6]{Cha20}). Note that $A_{\theta + tZ}$ and $A_\theta$ define the same noncommutative torus. We then have the following theorem.

\begin{theorem}\label{thm:main2}
	The K-theory classes of the fundamental projective modules $\mathcal{E}_{I}^{\theta + tZ},$ along with $[1]$ generate $\K_0(A_{\theta + tZ}),$ and hence $\K_0(A_\theta).$
\end{theorem}

\begin{proof}
	See \cite[Theorem 4.7]{Cha20}.
\end{proof}

\section{Noncommutative orbifolds and projective modules}\label{sec:orbifolds_projective}

Let us recall Example~\ref{ex:orbifold}. Let $W:=(a_{ij})$ be an invertible $n \times n$ matrix of finite order with integer entries and $F$ be the finite cyclic group generated by $W$. In addition, we assume that $W^t\theta W=\theta$. Hence $F$ is a finite subgroup of $\Sp(n, \Z, \theta):= \{ A \in \GL(n,\Z) : A^T\theta A = \theta$\}. By Lemma 2.1 of \cite{ELPW10} we have $C^*(\Z^n\rtimes F, \omega_\theta ') = A_{\theta} \rtimes_\alpha F,$ where the action of $F$ on $A_\theta$ is given by (see \cite[Equation 2.6]{JL15}):
 \begin{equation}\label{eq:actionontori}
 	\alpha(U_i)= e(\sum_{k=2}^n\sum_{j=1}^{k-1}a_{ki}a_{ji}\theta_{jk})U_1^{a_{1i}}\cdots U_n^{a_{ni}},
 \end{equation} where $U_1,...,U_n$ are the generators of $A_\theta$. Sometimes we just write the crossed product as $A_{\theta} \rtimes F,$ without the  $``\alpha"$ decoration. 

Let us look into the case where $n=2$. Note that $\Sp(2, \Z, \theta) = \SL(2, \Z)$. Finite cyclic subgroups of $\SL(2,\Z)$ are up to conjugacy generated by the following 4 matrices:  
$$W_{(2)}
  := \left( \begin{array}{ccc}
-1 & 0\\
  0  & -1\\
  \end{array} \right) , \:  
	W_{(3)}
  := \left( \begin{array}{ccc}
-1 & -1\\
  1  & 0\\
  \end{array} \right), 
  $$
  
  $$ W_{(4)}
  := \left( \begin{array}{ccc}
0 & -1\\
  1  & 0\\
  \end{array} \right),
  \:  W_{(6)}
  := \left( \begin{array}{ccc}
0 & -1\\
  1  & 1\\
  \end{array} \right),
   $$ where the notation $W_{(r)}$ indicates that it is a matrix of order $r$. The actions of the cyclic groups generated by these matrices are considered already in \cite{ELPW10}, where the authors constructed projective modules over the corresponding crossed products using the fundamental projective modules. 
   
   For $n\ge 3$ finding a finite order matrix $W\in \Sp(n, \Z, \theta)$ is non-trivial.  In \cite{JL15}, and in \cite{He19}, the authors found some of the matrices for $n \ge 3$ and studied the associated actions. Note that, for all $n$ there will always be a matrix $W$ of order 2, i.e. $-\id_n.$ The action by $\Z_2=\langle W \rangle$ is the flip action, which was already defined in the introduction.

   One natural question is how does one extend the fundamental projective modules, $\mathcal{E}_{I}^{\theta}$, over noncommutative tori $A_\theta$ to the aforementioned crossed products. In \cite{CL19}, this was answered when the module $\mathcal{E}_{I}^{\theta}$ is a completion of $\mathscr{S}(\R^p),$ i.e. when the dimension of the torus is even ($=2p$), and $\theta$ is invertible so that $\mathcal{E}_{I}^{\theta}$ is defined. This module is called the \emph{Bott class}. In this section we do this extension for a general $\mathcal{E}_{I}^{\theta}.$ We need the following proposition for such extensions. 
   \begin{proposition}\label{mainprop}
	Suppose $F$ is a finite group acting on a C*-algebra $A$ by the action $\alpha$. Also suppose that $\mathcal{E}$  is a finitely generated projective (right) $A$-module with a right action $T : F \rightarrow \Aut(\mathcal{E})$, written $(\xi, g) \rightarrowtail \xi T_g$, such that $\xi (T_g)a = (\xi\alpha_g(a)) T_g$ for all $\xi \in  \mathcal{E}, a \in A,$ and $g \in F$. Then $\mathcal{E}$ becomes a finitely generated projective $A\rtimes F$ module with action defined by
	
$$ \xi \cdot (\sum_{g \in F} a_g\delta_g) =  \sum_{g \in F} (\xi a_g)T_g.$$ 

Also, if we restrict the new module to $A$, we get the original $A$-module $\mathcal{E}$, with the action of $F$ forgotten. 
\end{proposition}

\begin{proof}
	This is exactly the construction of the Green--Julg map. See \cite[Proposition 4.5]{ELPW10}.
\end{proof}

Let us first recall the approach of \cite{CL19}, where the authors define the necessary action of $F$ on the Bott class which allows them to conclude that the Bott class is a projective module over the crossed product $A_\theta \rtimes F,$ using Proposition~\ref{mainprop}.  Hence assume $n$ ($=2p$) to be even for the moment. Since $F = \langle W \rangle$ acts on $\Z^n$ as before, we have $W^t\theta W = \theta.$ In order to define an action of $F$ on the Bott class, the authors (in \cite{CL19}) used the so-called \emph{metaplectic representation} of the symplectic matrix $TWT^{-1}$, where  $T^tJ_0T=\theta$ as in Equation~\ref{eq:Tmap}. Note that, in this case $q=0$ and hence $T=T_{11}.$ The main idea is to use the following metaplectic extension: 
\begin{center}
\begin{equation}\label{eq:metaplectic}
	\begin{tikzcd}
  0 \arrow{r} & \mathbb{S}^1 \arrow{r} & \Mp^c(n) \arrow{r}& \Sp(n)           \arrow{r} & 0 
    \end{tikzcd}
    \end{equation}
    	\end{center}
where $\Sp(n)$ is the usual symplectic group, and $\Mp^c(n)$ is the complex metaplectic group (see \cite[Section 5]{CL19}). For our purposes we do not need much details about the metaplectic group, but we need to know that it has a (metaplectic) representation on $\mathscr{S}(\R^p)$ (\cite[Definition 5.1]{CL19}, also see \cite[Chapter 7]{dG11}). 
Now, $F \cong \langle TWT^{-1} \rangle$ sits inside $\Sp(n).$ But also we have the following lift (since $\mathrm{H}^2(F,\mathbb{S}^1)$ is trivial, see \cite[page 158]{CL19}) possible:

 \begin{center}

\begin{tikzcd}
                &                                 &                                   & F  \arrow[ld, dashrightarrow]               \arrow{d}          &   \\
    0 \arrow{r} & \mathbb{S}^1 \arrow{r}   & \Mp^c(n)                 \arrow{r}        & \Sp(n)           \arrow{r} & 0 \\
    
\end{tikzcd}	
    \end{center}
The above defines an action of $F$ on $\mathscr{S}(\R^p)$ which extends to the necessary completion (Bott class) of $\mathscr{S}(\R^p)$ and it satisfies the conditions of Proposition~\ref{mainprop} (see \cite[Theorem 5.4]{CL19}). Hence the Bott class becomes a projective module over $A_\theta \rtimes F.$  In the following, we shall often write $fW$ for the above action of $W$ on $\mathscr{S}(\R^p)$, for $f \in \mathscr{S}(\R^p)$. So from \cite[Equation 5.12]{CL19} we have  

\begin{equation}\label{eq:CL_main}
	\left(f W\right) U_{l}=(f \alpha_{W}\left(U_{l}\right))W, \quad f \in \mathscr{S}(\R^p), l \in \Z^{2p},
	\end{equation}
which is the condition in Proposition~\ref{mainprop}.

Now we take a general $n$, not necessarily even. We have $\mathcal{R}:=\left\langle \rho(R), R \in GL(n,\mathbb{Z})\right\rangle \subseteq \mathrm{SO}(n, n|\mathbb{Z}).$ Also for $W \in F$ we have $\rho(W^t) \in \mathrm{SO}(n, n|\mathbb{Z}).$ In this way $F \subseteq \mathrm{SO}(n, n|\mathbb{Z}).$
Recall $g_{I, \Sigma}=\sigma_{2 p} \rho\left(R_{I}^{\Sigma}\right).$

\begin{theorem}\label{thm:mainW}
 With all the notations introduced before, assume $\pf(M_{I}^{\theta}) \neq 0.$ Let $W\in \GL(n,\Z)$ be of finite order such that $W^t\theta W = \theta$ and $F:=\langle W \rangle.$  Suppose $g_{I, \Sigma} F(g_{I, \Sigma})^{-1} \subset \mathcal{R}$ inside $\mathrm{SO}(n, n|\mathbb{Z}).$ Then  $\mathcal{E}_{I}^{\theta}$ becomes a finitely generated, projective module over $A_{\theta} \rtimes F$.
\end{theorem}

\begin{proof}
	
$g_{I, \Sigma} F(g_{I, \Sigma})^{-1} \subset \mathcal{R}$ means $\sigma_{2 p}  \rho\left(R_{I}^{\Sigma} W^t (R_{I}^{\Sigma})^{-1} \right)\sigma_{2 p} \in \mathcal{R},$ noting that the inverse of $\sigma_{2 p}$ is $\sigma_{2 p}$ again. Now $\left(R_{I}^{\Sigma} W^t (R_{I}^{\Sigma})^{-1} \right)^t$ is a $\rho(R_{I}^{\Sigma})\theta$-symplectic matrix and the algebras $A_{\rho(R_{I}^{\Sigma})\theta}$ and $A_{\theta}$ are  $F$-equivariantly isomorphic, where the action of $F$ on $A_{\rho(R_{I}^{\Sigma})\theta}$ is given by identifying $F$ with $\langle \left(R_{I}^{\Sigma} W^t (R_{I}^{\Sigma})^{-1} \right)^t \rangle$. So by passing from $W^t$ to $R_{I}^{\Sigma} W^t (R_{I}^{\Sigma})^{-1}$ if necessary, we may assume without loss of generality that $\mathcal{E}_{I}^{\theta}$ as an $A_{\theta}$-module. 

From $\sigma_{2 p} \rho\left(W^t\right)\sigma_{2 p} \in \mathcal{R},$ we have
	
	 \begin{equation} 
	 \label{eq:W_symplectic} 
		\sigma_{2 p} \rho\left(W^t\right)\sigma_{2 p}  = \left( \begin{array}{ccc}S & 0\\
  0  & (S^{-1})^t\\
 \end{array} \right), 
\end{equation} for some $S \in GL(n,\mathbb{Z})$. Writing $ W = \left( \begin{array}{ccc}
W_1 & W_2\\
  W_3  & W_4\\
 
\end{array} \right),
$ where $W_1$ is the $2p \times 2p$ block, a simple computation shows that  $W_2 = W_3 = 0$, and $ S = \left( \begin{array}{ccc}
W_1^{-1} & 0\\
  0  & W_4^t\\
 \end{array} \right).
$
So we have \begin{equation}
\tag{**}
\label{eq:form_of_W}
	 W = \left( \begin{array}{ccc}
W_1 & 0\\
  0  & W_4\\
 
\end{array} \right).
\end{equation}
Writing $\theta
  = \left( \begin{array}{ccc}
\theta_{11} & \theta_{12}\\
  \theta_{21}  & \theta_{22}\
  \end{array} \right)$ as before, $W^t\theta W = \theta$ gives the following \emph{compatibility} relations:
  \begin{equation}\label{eq:comptability}
  	\begin{cases}
&W_1^t\theta_{11} W_1= \theta_{11}\\
 & W_1^t\theta_{12} W_4= \theta_{12}\\
  & W_4^t\theta_{21} W_1= \theta_{21}\\
   & W_4^t\theta_{22} W_4= \theta_{22}
\end{cases}
 \end{equation}

Let us first write down Equation~\ref{eq:proj_mod_T}, which is

\begin{equation} \label{eq:proj_mod_T_}
(f U_l)(x)= e(\langle -T(l), J^{\prime} T(l)/2\rangle)\langle x, T^{\prime \prime}(l) \rangle f(x-T^{\prime}(l)), 
\end{equation} 
 more explicitly. Writing $l=(l_1,l_2)\in \Z^n,$ for $l_1\in \Z^{2p}, l_2 \in Z^q,$ we have, $$
 T(l) = \left( \begin{array}{cc}
T_{11} & 0\\
0  & \id_q\\
\theta_{21}  & \frac{\theta_{22}}{2}
\end{array} \right)\left( \begin{array}{c}
l_1\\
l_2
\end{array} \right)=\left( \begin{array}{cc}
T_{11}l_1\\
l_2\\
\theta_{21}l_1 + \frac{\theta_{22}}{2}l_2
\end{array} \right).
$$ 
Let $J'_0$ be the matrix obtained by replacing the negative entries of $J_0$ by zeroes. Also, if $Q'$ and $Q''$ be the canonical projections of $\R^p \times \widehat{\R^p}$ to $\R^p$ and $\widehat{\R^p}$, respectively, denote
$$T'_{11}:=Q'\circ T_{11}, \quad T''_{11}:=Q''\circ T_{11}.$$
Then
\begin{align*}
e(\langle -T(l), J^{\prime} T(l)/2\rangle) &= e\left(-\left( \begin{array}{cc}
T_{11}l_1\\
l_2\\
\theta_{21}l_1 + \frac{\theta_{22}}{2}l_2
\end{array} \right)\cdot \left( \begin{array}{ccc}
J'_0 & 0 & 0\\
0 & 0 & \id_q\\
0 & 0 & 0\\
\end{array} \right)  \left( \begin{array}{cc}
T_{11}l_1\\
l_2\\
\theta_{21}l_1 + \frac{\theta_{22}}{2}l_2
\end{array} \right)/2 \right) 
%\intertext{where we used the relation $a^2-b^2=(a+b)(a-b)$, and we can continue:}
\\ &= e\left(-\left( \begin{array}{cc}
T_{11}l_1\\
l_2\\
\theta_{21}l_1 + \frac{\theta_{22}}{2}l_2
\end{array} \right)\cdot \left( \begin{array}{cc}
J'_0T_{11}l_1\\
\theta_{21}l_1 + \frac{\theta_{22}}{2}l_2\\
0
\end{array} \right)/2 \right) 
\\ &= e\left(-T_{11}l_1\cdot J'_0T_{11}l_{1}/2 \right)e\left(-l_2\cdot \theta_{21}l_1-l_2\cdot \frac{\theta_{22}}{2}l_2 /2 \right)
\\ &= C_1(l_1) A(l_1,l_2),
\end{align*}
where $C_1(l_1):=e\left(-T_{11}l_1\cdot J'_0T_{11}l_{1}/2 \right), \quad A(l_1,l_2):= e\left(-l_2\cdot \theta_{21}l_1-l_2\cdot \frac{\theta_{22}}{2}l_2 /2 \right).$
\begin{multicols}{2}
\begin{align*}
\langle x, T^{\prime \prime}(l) \rangle &= \left\langle \left( \begin{array}{c}
x_1\\
x_2
\end{array} \right), \left( \begin{array}{c}
T''_{11}l_1\\
\theta_{21}l_1 + \frac{\theta_{22}}{2}l_2
\end{array} \right) \right\rangle 
%\intertext{where we used the relation $a^2-b^2=(a+b)(a-b)$, and we can continue:}
\\ &= \left\langle x_1,
T''_{11}l_1 \right\rangle \left\langle x_2, \theta_{21}l_1 + \frac{\theta_{22}}{2}l_2
 \right\rangle
 \\ &= C_2(x_1, l_1) B(x_2,l_1,l_2),
\end{align*}
\begin{align*}
\\
\\
f(x-T^{\prime}(l)) &= f\left( \begin{array}{c}
x_1-T'_{11}l_1\\
x_2-l_2
\end{array}\right),
\end{align*}
\end{multicols}
\noindent where $C_2(x_1,l_1):= \left\langle x_1,
T''_{11}l_1 \right\rangle, \quad B(x_2,l_1,l_2):= \left\langle x_2, \theta_{21}l_1 + \frac{\theta_{22}}{2}l_2
 \right\rangle.$

Now, for $f\in \mathscr{S}(\R^p \times \Z^q),$ we define 
\begin{equation}\label{eq:meta_gen}
(fW)(x_1,x_2):= \sqrt{\det (W_4)}(f^\sharp W_1)(x_1),
\end{equation}
 where $ f^\sharp \in \mathscr{S}(\R^p)$ defined as $f^\sharp(x')=f(x',W_4x_2).$ Note that here we have used the metaplectic action of $W_1$ on $f^\sharp.$ We first want to show that $f\rightarrow fW$ extends to a unitary operator on $L^2(\R^p \times \Z^q)$ using the fact that the metaplectic operators are unitary. To this end we check that 
\begin{equation}\label{eq:action_unitary}
	\langle fW, g \rangle_{L^2} =  \langle f, gW^{-1} \rangle_{L^2},
	\end{equation}
	which follows from the following computation.
\begin{align*}
\langle fW, g \rangle_{L^2} &= \int_G(fW)(x_1,x_2)\overline{g(x_1,x_2)}dx
%\intertext{where we used the relation $a^2-b^2=(a+b)(a-b)$, and we can continue:}
\\ &=  \sqrt{\det (W_4)} \int_G(f^\sharp W_1)(x_1)\overline{g(x_1,x_2)}dx
\\ &=  \sqrt{\det (W_4)} \int_G f^\sharp(x_1)\overline{(g'W_1^{-1})(x_1)}dx_1dx_2 \quad (\text{where $g'(x_1)=g(x_1,x_2)$})
\\ &=  \sqrt{\det (W_4)}\det (W_4)^{-1} \int_G f(x_1,x_2)\overline{g^\sharp W_1^{-1} (x_1)}dx_1dx_2 \quad (\text{change $x_2$ to $W_4^{-1}x_2$})
\\ &=  \int_G f(x_1,x_2)\overline{gW^{-1}(x_1,x_2)}dx_1dx_2
\\ &=   \langle f, gW^{-1} \rangle_{L^2}.
\end{align*}
We want to show 
\begin{equation}\label{main-general_1}
(fW)U_l = (f\alpha_{W}(U_l))W = (fU_{Wl})W.
	\end{equation}
	From Equation~\ref{eq:CL_main} we already have
\begin{equation}\label{main-general}
(f^\sharp W_1)U_{l_1} = (f^\sharp U_{W_1l_1})W_1.
	\end{equation}

Now
\begin{align*}
(f U_{Wl})^\sharp (x_1) &= (f U_{Wl})(x_1,W_4x_2),
%\intertext{where we used the relation $a^2-b^2=(a+b)(a-b)$, and we can continue:}
\\ &= C_1(W_1l_1) A(W_1l_1,W_4l_2) C_2(x_1, W_1l_1)  B(W_4x_2,W_1l_1,W_4l_2) f\left( \begin{array}{c}
x_1-T'_{11}W_1l_1\\
W_4x_2-W_4l_2
\end{array}\right)
\\ &\overset{\ref{eq:comptability}}= C_1(W_1l_1) C_2(x_1, W_1l_1) A(l_1,l_2) B(x_2,l_1,l_2)f\left( \begin{array}{c}
x_1-T'_{11}W_1l_1\\
W_4(x_2-l_2)
\end{array}\right)
\\ &= A(l_1,l_2) B(x_2,l_1,l_2)(g^\sharp U_{W_1l_1})(x_1), \hspace{.2cm} \text{where} \hspace{.2cm} g(x_1)=f(x_1, x_2-l_2).
\end{align*}
So the RHS of \ref{main-general_1} becomes 
\begin{align*}
\sqrt{\det (W_4)}((f U_{Wl})^\sharp) W_1 (x_1)&= \sqrt{\det (W_4)}A(l_1,l_2) B(x_2,l_1,l_2)((g^\sharp U_{W_1l_1})W_1)(x_1)
%\intertext{where we used the relation $a^2-b^2=(a+b)(a-b)$, and we can continue:}
\\ &\overset{\ref{main-general}}= \sqrt{\det (W_4)}A(l_1,l_2) B(x_2,l_1,l_2)(g^\sharp W_1)U_{l_1}(x_1).
\end{align*}
Now the LHS 
\begin{align*}
(fW)U_l(x) &= C_1(l_1) C_2(x_1, l_1) A(l_1,l_2) B(x_2,l_1,l_2)fW\left( \begin{array}{c}
x_1-T'_{11}l_1\\
x_2-l_2
\end{array}\right) 
\\ &= \sqrt{\det (W_4)}C_1(l_1) C_2(x_1, l_1) A(l_1,l_2) B(x_2,l_1,l_2)(g^\sharp W_1)(x_1-T'_{11}l_1) 
\\ &= \sqrt{\det (W_4)}A(l_1,l_2) B(x_2,l_1,l_2)(g^\sharp W_1)U_{l_1}(x_1).
\end{align*}
Thus we have proved Equation~\ref{main-general_1}. We finish the proof with the compatibility of the action with the inner product $\langle .,. \rangle_{\mathcal{A}^\infty}$ as defined in \eqref{eq:proj_mod_inner}:
\begin{equation*}
  \langle fW,gW \rangle_{\mathcal{A}^\infty}=\alpha_{W^{-1}}(\langle f,g \rangle_{\mathcal{A}^\infty}).
\end{equation*}
This will make sure that the action of $F$ on $\mathscr{S}(\R^p \times \Z^q)$ defined through Equation~\ref{eq:meta_gen} has a unique extension to $\mathcal{E}_{I}^{\theta},$ and hence we can use Proposition~\ref{mainprop}. Now replacing $f$ by $fW^{-1}$, it suffices to check:
\begin{equation}\label{eq:comp_inner}
   \langle f,gW \rangle_{\mathcal{A}^\infty}=\alpha_{W^{-1}}(\langle fW^{-1},g \rangle_{\mathcal{A}^\infty}).
\end{equation}
Note that  
\begin{equation}\label{eq:inner}
\langle f,g \rangle_{\mathcal{A}^\infty}(l)=\langle gU_{-l},f\rangle_{L^2}%=\langle g, fU^\theta_l\rangle_{L^2},
\end{equation}
for $\langle f,g\rangle_{L^2}=\int_{G}f(x)\overline{g(x)}dx$, and hence
\begin{equation}\label{eq:inner2}
\alpha_{W^{-1}}(\langle f,g \rangle_{\mathcal{A}^\infty})(l)=\langle g\alpha_{W}(U^\theta_{-l}),f\rangle_{L^2}.%=\langle g,f\alpha_{W^{-1}}(U_l) \rangle_{L^2}.
\end{equation}
Now
\begin{align*}
 \langle f,gW \rangle_{\mathcal{A}^\infty}(l) &\overset{\ref{eq:inner}}= \langle (gW)U_{-l},f\rangle_{L^2}
\\ &\overset{\ref{main-general_1}}= \int_{\R^p \times \Z^q} (g\alpha_{W}(U_{-l})W(x))\overline{f(x)}dx,
\\ &\overset{\ref{eq:action_unitary}}= \int_{\R^p \times \Z^q} (g\alpha_{W}(U_{-l}))(x)\overline{(fW^{-1})(x))}dx,
\\ &\overset{\ref{eq:inner2}}= \alpha_{W^{-1}}(\langle fW^{-1} ,g\rangle_{\mathcal{A}^\infty})(l).
\end{align*}
which is the desired identity.
\end{proof}

\begin{remark}
\normalfont
	The condition of the above theorem holds if and only if $R_{I}^{\Sigma} W^t (R_{I}^{\Sigma})^{-1}$ is of the form \ref{eq:form_of_W}. However, the condition reveals more information: Equation~\ref{eq:W_symplectic} really means that one can define an action of $\langle S^t \rangle$ on $A_{\sigma_{2p}\theta},$ and the Morita equivalence between $A_\theta$ and $A_{\sigma_{2p}\theta}$ can be lifted to an equivalence between the corresponding orbifolds. This will appear in a joint work with Ullisch (\cite{CU21}). 
\end{remark}

As an immediate corollary we have, 

\begin{corollary}\label{cor:mainW}
 Let $Z$ and $t$ be as in Subsection~\ref{sec:generators}. With all the notations introduced before, assume $\pf(M_{I}^{\theta+tZ}) \neq 0.$ Let $W\in \GL(n,\Z)$ be of finite order such that $W^t\theta W = \theta, W^tZW=Z$ and let $F:=\langle W \rangle.$  Suppose $g_{I, \Sigma} F(g_{I, \Sigma})^{-1} \subset \mathcal{R}$ inside $\mathrm{SO}(n, n|\mathbb{Z}).$ Then  $\mathcal{E}_{I}^{\theta +tZ}$ becomes a finitely generated, projective module over $A_{\theta} \rtimes F$.
\end{corollary}
\begin{proof}
Follows immediately from the preceding theorem and Theorem~\ref{thm:main2}, noting that the isomorphism between $A_{\theta + tZ}$ and $A_{\theta}$ is F-equivariant.	
\end{proof}
\noindent The above corollary shows that under the extra assumption $W^tZW=Z,$ all the fundamental projective modules over $A_{\theta +tZ}\cong A_{\theta}$ become finitely generated, projective modules over the crossed product.

\section{Some applications: Morita equivalence of noncommutative orbifolds}\label{Sec:ntflip}

\subsection{Trace of the extended module}

Let $F$ be a finite group acting on a C*-algebra $A$. Also suppose that $\tau$ is an $F$-invariant trace on $A$. Then we can define a trace $\tau^F$ on $A\rtimes F$ by 

$$  \tau^F (\sum_{g \in F} a_g\delta_g) := \tau(a_e). $$
Let $\mathcal{E}^F$ denote the finitely generated, projective $A\rtimes F$-module, which is obtained from a finitely generated, projective $A$-module $\mathcal{E},$ as in Proposition~{\ref{mainprop}}.

\begin{lemma}\label{lem:trace}
	
$\tau^F ([\mathcal{E}^F]) = \frac{\tau ([\mathcal{E}])}{|F|},$ where $[\mathcal{E}^F]$ and $[\mathcal{E}]$ denote the K-theory classes of $[\mathcal{E}^F]$ and $[\mathcal{E}],$ respectively. 

\end{lemma}

\begin{proof}
	Let $p^F$ denote the projection corresponding to $\mathcal{E}^F,$ and $p$ the projection corresponding to $\mathcal{E}.$ Define the canonical injection (regular representation) $\Psi$ from $A\rtimes F$ to $A \otimes \mathcal{B}\left(l^{2}(F)\right)$ by mapping $a$ to $\sum_{g \in F} g \cdot a \otimes p_{g}$ (where $p_{g}$ is the projection onto the functions supported on $\{g\}$) and by mapping $\delta_g$ to $1 \otimes \rho(g),$ where $\rho$ is the right regular representation. It is well known that the above map defines an inverse to the Green--Julg map in $F$-equivariant K-theory (see e.g. \cite[page 191]{HG04}). If $p^F$ is in $A\rtimes F$, let us write $p^F= \sum_{g \in F} a_g\delta_g.$ Then $\tau^F([p^F])= \tau(a_e).$ On the other hand, $[\Psi (p^F)] = [p]$ in $\K(A \otimes \mathcal{B}\left(l^{2}(F)\right)),$ and hence, $\tau([\Psi (p^F)]) = \tau([p]).$ But $\tau([\Psi (p^F)]) = |F| \tau(a_e),$ using the above formula of $\Psi$ and the fact that $\tau$ is $F$-invariant. Hence $\tau^F([p^F])= \frac{\tau([p])}{|F|}$. A similar computation holds when $p^F$ is in some matrix algebra over $A\rtimes F$.

\end{proof}

\subsection{Images of the canonical traces of noncommutative orbifolds}

Let us come back to the noncommutative orbifolds. As in Example~\ref{ex:orbifold}, take a finite order matrix $W\in \GL(n,\Z)$ such that $W^t\theta W = \theta$. Assume that the order of $W$ is $N.$ We then have $C^*(\Z^n\rtimes F, \omega_\theta ') = A_{\theta} \rtimes_\alpha F,$ $F:=\langle W \rangle.$ 

For $A_{\theta} \rtimes_\alpha F,$ the regular representation $\Psi:A_{\theta} \rtimes_\alpha F \hookrightarrow \mathrm{M}_N(A_{\theta})$ is given by the following:

\begin{equation}
	\Psi(\sum_{i=0}^{N-1} a_i W^i) = \left(\begin{array}{ccccc}a_{0} & a_{1} & a_{2} & \cdots & a_{N-1} \\ \alpha\left(a_{N-1}\right) & \alpha\left(a_{0}\right) & \alpha\left(a_{1}\right) & \cdots & \alpha\left(a_{N-2}\right)\\ \alpha^{2}\left(a_{N-2)}\right. & \alpha^{2}\left(a_{N-1}\right) & \alpha^{2}\left(a_{0}\right) & \cdots & \vdots \\ \vdots & \ddots & \ddots & \ddots & \alpha^{N-2}\left(a_{1}\right) \\ \alpha^{n-1}\left(a_{1}\right) & \alpha^{N-1}\left(a_{2}\right) & \cdots & \alpha^{N-1}\left(a_{N-1}\right) & \alpha^{N-1}\left(a_{0}\right)\end{array}\right).
\end{equation} 
The canonical trace $\Tr$ on $A_{\theta}$ is clearly $F$-invariant. Now the canonical trace on $A_{\theta} \rtimes F$ is given by 

$$\Tr^F(\sum_{i=0}^{N-1} a_i W^i):= \Tr(a_0).$$ If we identify $A_{\theta} \rtimes F$ inside $\mathrm{M}_N(A_{\theta})$ via the map $\Psi,$ the trace $\Tr^F$ is the normalised trace on $\mathrm{M}_N(A_{\theta}).$ This immediately gives 

\begin{equation}\label{eq:trace_orbifold}
	\Tr^F(\K_0(A_{\theta} \rtimes F))\subseteq \frac{1}{N}\Tr(\K_0(A_{\theta})). 
\end{equation}
So from Equation~\ref{eq:image_of_trace_nt} we have 
$$
	\Tr^F(\K_0(A_{\theta} \rtimes F))\subseteq \frac{1}{N}\Bigg( \mathbb{Z}+\sum_{0<|I| \leq n} \pf(M_{I}^{\theta}) \mathbb{Z}\Bigg).
$$
Our main theorem, Theorem~\ref{thm:mainW}, gives sufficient conditions on $W$ so that $\frac{1}{N}\pf(M_{I}^{\theta}) \in \Tr^F(\K_0(A_{\theta} \rtimes F))$ as we have the following theorem.

\begin{theorem}\label{th:main_trace}
 With all the notations introduced before, let $W\in \GL(n,\Z)$ be of finite order such that $W^t\theta W = \theta$ and $F:=\langle W \rangle.$  Suppose $g_{I, \Sigma} F(g_{I, \Sigma})^{-1} \subset \mathcal{R}$ inside $\mathrm{SO}(n, n|\mathbb{Z}).$ Then $\frac{1}{N}\pf(M_{I}^{\theta}) \in \Tr^F(\K_0(A_{\theta} \rtimes F)),$ where $N$ is the order of $W$. 
	\end{theorem}
	
	\begin{proof}
		If $\pf(M_{I}^{\theta}) \neq 0,$ using Theorem~\ref{thm:mainW}, $\mathcal{E}_{I}^{\theta}$ becomes a projective $A_{\theta} \rtimes F$-module.  Since the trace of $\mathcal{E}_{I}^{\theta}$ is $\pf(M_{I}^{\theta}),$ use Lemma~\ref{lem:trace}. If $\pf(M_{I}^{\theta}) = 0,$ the statement is obvious.
	\end{proof}

We now discuss various examples. We observe that the trace of the projection $p_0 := \frac{1}{N}(1 + W + W^2 + \cdots +W^{N-1}) \in \K_0(A_{\theta} \rtimes F)$ is $\frac{1}{N}$. Also for an even $n$ with $n=2p$, and $I:=(1, 2, \cdots, 2p),$ $\Sigma$ must be trivial. In this case $g_{I, \Sigma} F(g_{I, \Sigma})^{-1} \subset \mathcal{R},$ since $g_{I, \Sigma} \rho(W^t) (g_{I, \Sigma})^{-1}= \rho(W^{-1}).$  Hence $\frac{1}{N}\pf(M_{I}^{\theta}) = \frac{1}{N}\pf(\theta)\in \Tr^F(\K_0(A_{\theta} \rtimes F)),$ for $|I|=n=2p.$ For an odd $n$, we of course have  $\frac{1}{N}\pf(\theta)=0\in \Tr^F(\K_0(A_{\theta} \rtimes F)).$   

\begin{example}\label{ex:trace_2_dim}(2-dimensional cases)
\normalfont
	Let $\theta$ be a real number. For the 2-dimensional torus $A_\theta,$ we have actions of $F:=\langle W \rangle,$ where $W=W_{(2)}, W_{(3)}, W_{(4)}, W_{(6)}\in \SL(2,\Z),$ on $A_\theta$ as in Section \ref{sec:orbifolds_projective}. From the above observations, we have $\frac{1}{N}$ and $\frac{1}{N}\pf\left(\left( \begin{array}{cc}
0 & \theta \\
-\theta & 0 \\
\end{array} \right) \right) \in \Tr^F(\K_0(A_{\theta} \rtimes F)),$ where $N=2,3,4,6$ for $W=W_{(2)}, W_{(3)}, W_{(4)}, W_{(6)},$ respectively. But $\pf\left(\left( \begin{array}{cc}
0 & \theta \\
-\theta & 0 \\
\end{array} \right) \right)  = \theta.$ Hence

	\begin{equation}\label{eq:trace_2dim} \Tr^F(\K_0(A_{\theta} \rtimes F))= \frac{1}{N}( \mathbb{Z}+\theta \mathbb{Z})=\frac{1}{N}\Tr(\K_0(A_{\theta})).\end{equation}
	In these 2-dimensional cases, the above ranges $\Tr^F(\K_0(A_{\theta} \rtimes F))$ have  already been computed in \cite{ELPW10}.
	
\end{example}

\begin{example}\label{ex:diagonal} (diagonal actions on 4-dimensional tori)
\normalfont
	First take $\theta_1 \in \mathcal{T}_{n_1}(\R)$ and $\theta_2 \in \mathcal{T}_{n_2}(\R)$. Let $W_1, W_2$ be $\theta_1$-symplectic and $\theta_2$-symplectic matrices of order $N_1$ and $N_2,$ respectively.
	Then clearly $W:= \left( \begin{array}{cc}
W_1 & 0 \\
0 & W_2 \\
\end{array} \right) 
$ is a $\theta:= \left( \begin{array}{cc}
\theta_1 & 0 \\
0 & \theta_2 \\
\end{array} \right) 
$-symplectic matrix of order of order $N:=lcm(N_1, N_2).$ Hence $F:=\langle W \rangle$ acts on $A_\theta.$ Clearly $\pf (\theta), \pf(\theta_1), \pf(\theta_2)$ belong to $\Tr(\K_0(A_{\theta})).$ Now assume that these three terms are non-zero. So $n_1$ and $n_2$ must be even. Then $\frac{\pf (\theta)}{N}$ is in $\Tr^F(\K_0(A_{\theta} \rtimes F)),$ from the previous observation.  For, $I=(1,2,\cdots, n_1)$ and $I=(n_1+1, n_1+1, \cdots, n_1+n_2)$, one can choose $R_{I}^{\Sigma} = \id_{n_1+n_2}$ and $ \left( \begin{array}{cc}
0 & \id_{n_1} \\
\id_{n_2} & 0 \\
\end{array} \right) 
$, respectively. Then one easily checks that in both cases, $g_{I, \Sigma} F(g_{I, \Sigma})^{-1} \subset \mathcal{R}.$ Hence $\frac{\pf (\theta_1)}{N}$, $\frac{\pf (\theta_2)}{N}$ are in $\Tr^F(\K_0(A_{\theta} \rtimes F)).$

 Let us specialise this example to $n_1=n_2=2.$ The $4 \times 4$ matrix $\theta$ is then given by $$
 \theta = \left( \begin{array}{cccc}
0 & \theta_{12} & 0 & 0\\
-\theta_{12} & 0 & 0 & 0\\
0 & 0 & 0 & \theta_{34} \\
0 & 0 & -\theta_{34} & 0 \\
\end{array} \right) .  
$$
In this case $\Tr(\K_0(A_{\theta}))=\mathbb{Z}+\sum_{0<|I| \leq 4} \pf\left(M_{I}^{\theta}\right) \mathbb{Z} = \mathbb{Z}+ \pf\left(M_{(1,2)}^{\theta}\right) \mathbb{Z}+ \pf\left(M_{(3,4)}^{\theta}\right) \mathbb{Z}+ \pf\left(M_{(1,2, 3, 4)}^{\theta}\right) \mathbb{Z} = \mathbb{Z}+ \theta_{12} \mathbb{Z}+ \theta_{34} \mathbb{Z}+ \theta_{12}\theta_{34} \mathbb{Z}.$ Let us also take $W_1, W_2 \in \SL(2,\Z)$ of finite order (say $N_1$ and $N_2,$ respectively) as in the previous example. Then $F:=\langle W \rangle$ acts on $A_\theta,$ where $W:= \left( \begin{array}{cc}
W_1 & 0 \\
0 & W_2 \\
\end{array} \right). 
$ Using the above, 
 
$$\Tr^F(\K_0(A_{\theta} \rtimes F))= \frac{1}{N}\Tr(\K_0(A_{\theta})),$$ for $N:=lcm(N_1, N_2).$  
One may look at \cite{He19} for more examples of a similar kind, where one can compute the ranges of the traces explicitly just like the above.

\end{example}

\begin{example}\label{ex:trace_flip}(flip actions on n-dimensional noncommutative tori)
\normalfont
	 Let us consider the flip action ($W= -\id_n$) of $\Z_2$ on an $n$-dimensional noncommutative torus $A_\theta$. In this case, $g_{I, \Sigma} \rho(W^t) (g_{I, \Sigma})^{-1}= \rho(W^{t}).$ Hence $g_{I, \Sigma} F(g_{I, \Sigma})^{-1} \subset \mathcal{R}$ trivially,  for every $I$ and $\Sigma$.
	Hence    
	\begin{equation}\label{eq:trace_flip}
		\Tr^{\Z_2}(\K_0(A_{\theta} \rtimes \Z_2))= \frac{1}{2}\Bigg( \mathbb{Z}+\sum_{0<|I| \leq n} \pf\left(M_{I}^{\theta}\right) \mathbb{Z}\Bigg)=\frac{1}{2}\Tr(\K_0(A_{\theta})). 
\end{equation}
	
\end{example}
Note that Example~\ref{ex:trace_2_dim} and Example~\ref{ex:trace_flip} also satisfy the conditions of Corollary \ref{cor:mainW}.

\subsection{Morita equivalence of noncommutative tori and orbifolds}\label{Sec:trace_orbi}

To obtain results about classification, we will restrict ourselves to simple C*-algebras. We start with the following definition.

\begin{definition}\label{def:nondegenerate}
\normalfont
	A skew symmetric real $n \times n$ matrix $\theta$ is called \emph{non-degenerate} if whenever $x \in \mathbb{Z}^{n}$ satisfies $e(\langle x, \theta y\rangle)=1$ for all $y \in \Z^n,$ then $x=0 .$ 
\end{definition}

 Let us denote the canonical trace of $A_\theta$ by $\Tr_{\theta}.$ We want to prove the following theorem.

\begin{theorem}\label{thm:main_Morita}
 	Let $\theta_{1}$ and $\theta_{2}$ be non-degenerate inside $\mathcal{T}_n$. Let $W\in \GL(n,\Z)$ be of finite order such that $W^t\theta_1 W = \theta_1$ and $W^t\theta_2 W = \theta_2$. Also assume that the action of $F:=\langle W \rangle$ on $\Z^n$ is free outside the origin $0 \in \Z^{n} .$ Then $A_{\theta_{1}}\rtimes F$ is strongly Morita equivalent to $A_{\theta_{2}}\rtimes F$ if and only if there exists a $\lambda>0$ such that $\Tr^F_{\theta_{1}}$ and $\lambda \Tr^F_{\theta_{2}}$ have the same range.
 	 \end{theorem}
 	 
 	 It is clear that the actions in Example~\ref{ex:trace_2_dim}, and Example~\ref{ex:trace_flip} and the 4-dimensional example in Example~\ref{ex:diagonal} are free outside the origin $0 \in \Z^{n} .$  Also in \cite{JL15} and \cite{He19}, various examples of $W$ are constructed which have the same property. 
 	 
 	 The proof of Theorem~\ref{thm:main_Morita} needs some preparation. Let us first recall the following proposition.
 	  \begin{proposition}\label{prop:phi06}(\cite[Proposition 3.7]{Phi06})
	 Let $A$ be a simple infinite dimensional separable unital nuclear C*-algebra with tracial rank zero and which satisfies the Universal Coefficient Theorem. Then $A$ is a simple $A H$ algebra with real rank zero and no dimension growth. If $\K_{*}(A)$ is torsion free, $A$ is an AT algebra. If, in addition, $\K_{1}(A)=0$, then $A$ is an $A F$ algebra.
\end{proposition}

Let $\theta \in \mathcal{T}_{n}$ be non-degenerate. Then the following are known.

\begin{itemize}
	\item $A_{\theta}$ is a simple C*-algebra (even the converse is true: simplicity of $A_\theta$ implies $\theta$ must be non-degenerate) with a unique tracial state (\cite[Theorem 1.9]{Phi06});
	\item $A_{\theta}$ has tracial rank zero (\cite[Theorem 3.6]{Phi06});
	\item If $\beta$ is an action of a finite group on $A_{\theta}$  which has the tracial Rokhlin property (see \cite[Section 5]{ELPW10}), $A_{\theta} \rtimes_{\beta} F$ is a simple C*-algebra with tracial rank zero (\cite[Corollary 1.6, Theorem 2.6]{Phi11}). Also, $A_{\theta} \rtimes_{\beta} F$ has a unique tracial state (\cite[Proposition 5.7]{ELPW10});
	\item Let $W\in \GL(n,\Z)$ be of finite order such that $W^t\theta W = \theta$. The the action $\alpha$ of $F:=\langle W \rangle$ on $A_\theta$ has the tracial Rokhlin property (\cite[Lemma 5.10 and Theorem 5.5]{ELPW10});
	\item For the action $\alpha,$ $A_{\theta} \rtimes_{\alpha} F$ satisfies the Universal Coefficient Theorem (\cite[Proposition 3.1]{JL15});
	\item $\K_*(A_{\theta} \rtimes_{\alpha} F) \cong \K_*(C^*(\Z^n\rtimes F, \omega_\theta ')) \cong \K_*(C^*(\Z^n\rtimes F))$ (\cite[Theorem 0.3]{ELPW10}).
\end{itemize}
For the K-groups of $C^*(\Z^n\rtimes F),$ the following result is known.

 	\begin{theorem} (\cite[Theorem 0.1]{LL12})\label{thm:LL12} Let $n, m \in \mathbb{N .}$ Consider the extension of groups $$1 \rightarrow \Z^n \rightarrow \Z^n \rtimes \Z_m \rightarrow \Z_m \rightarrow 1$$ such that the action of
$\Z_m$ on $\Z^n$ is free outside the origin $0 \in \Z^{n} .$ Then $\K_{0}\left(C^{*}\left(\Z^n \rtimes \Z_m\right)\right) \cong \mathbb{Z}^{s_{0}}$ for some
$s_{0} \in \mathbb{Z}$ and $\K_{1}\left(C^{*}\left(\Z^n \rtimes \Z_m\right)\right) \cong \mathbb{Z}^{s_{1}}$ for $s_{1} \in \mathbb{Z}.$ If $m$ is even, $s_{1}=0 .$

\end{theorem}

We are now ready to prove Theorem~\ref{thm:main_Morita}.

\begin{proof}[Proof of Theorem~\ref{thm:main_Morita}]

The freeness condition and the fact that $\K_*(A_{\theta} \rtimes F) \cong \K_*(C^*(\Z^n\rtimes F))$ show that the K-groups are torsion free, using Theorem~\ref{thm:LL12}. Then the above list of results along with Proposition~\ref{prop:phi06} shows that $A_{\theta_{1}}\rtimes F$ and $A_{\theta_{2}}\rtimes F$ are AT algebras.

	Assume there is a $\lambda>0$ such that $\Tr^F_{\theta_{1}}$ and $\lambda \Tr^F_{\theta_{2}}$ have the same range. Now it is enough to find an isomorphism $g: \K_{0}\left(A_{\theta_{1}}\rtimes F\right) \rightarrow \K_{0}\left(A_{\theta_{2}}\rtimes F\right)$ such that $\lambda\Tr^F_{\theta_{2}} \circ g=\Tr^F_{\theta_{1}}.$ Indeed, $g$ is then an order isomorphism by \cite[Proposition 3.7]{BCHL18}, and $g([1])\in \K_{0}\left(A_{\theta_{2}}\rtimes F\right)^+.$ So there is a $q \in \mathbb{N}$ and a projection $p \in \mathrm{M}_{q}\left(A_{\theta_{2}}\rtimes F\right)$ such that $[p]=g([1]) .$ Since $A_{\theta_{2}}\rtimes F$ is simple, $p$ is full.  Then $A_{\theta_{1}}\rtimes F$ and $p \mathrm{M}_{q}\left(A_{\theta_{2}}\rtimes F\right) p$ have isomorphic Elliott invariants of AT algebras, so $A_{\theta_{1}}\rtimes F \cong$ $p \mathrm{M}_{q}\left(A_{\theta_{2}}\rtimes F\right) p$ by classification (\cite{Ell93}, \cite[Theorem 5.2]{Lin04}). Clearly the right hand side algebra is Morita equivalent to $A_{\theta_{2}}\rtimes F.$
	
	Let us now see the existence of the isomorphism $g$. Denote the ranges of $\Tr^F_{\theta_{1}}$ and $\Tr^F_{\theta_{2}}$ by $R_1$ and $R_2,$ respectively. Since $R_1$ and  $R_2$ are finitely generated subgroups of $\R,$ they are free. Also $R_1 = \lambda R_2$ implies that they have the same rank. Now we have the following exact sequences :

\begin{center}

	\begin{tikzcd}
  0 \arrow[r] & \operatorname{ker}\left(\Tr^F_{\theta_{1}}\right) \arrow{r} & \K_{0}(A_{\theta_{1}}\rtimes F)  \arrow[r, "\Tr^F_{\theta_{1}}"]& R_1           \arrow{r} & 0 
    \end{tikzcd}
    
    	\end{center}
    	\begin{center}

	\begin{tikzcd}
  0 \arrow[r] & \operatorname{ker}\left(\Tr^F_{\theta_{2}}\right) \arrow{r} & \K_{0}(A_{\theta_{2}}\rtimes F)  \arrow[r, "\Tr^F_{\theta_{2}}"']& R_2           \arrow{r} & 0 
    \end{tikzcd}
    
    	\end{center}
Note that the above sequences split since the K-groups are torsion free. Now $\operatorname{ker}\left(\Tr^F_{\theta_{1}}\right)$ and $\operatorname{ker}\left(\Tr^F_{\theta_{2}}\right)$ are finitely generated abelian groups of the same rank. So there exists   an isomorphism $\psi$ between them. Now $g$ is defined as $\psi \oplus \phi,$ where $\phi$ is the map between $R_1$ and $R_2$ given by multiplication with $\frac{1}{\lambda}.$ Clearly $\lambda\Tr^F_{\theta_{2}} \circ g=\Tr^F_{\theta_{1}},$ since the following diagram commutes. 

	\begin{center}
\begin{tikzcd}
\K_{0}(A_{\theta_{1}}\rtimes F) \arrow[r, "\Tr^F_{\theta_{1}}"] \arrow[d, "g"]
& R_1 \arrow[d, "\phi"] \\
\K_{0}(A_{\theta_{2}}\rtimes F) \arrow[r, "\Tr^F_{\theta_{2}}"' ]
& R_2
\end{tikzcd}
	\end{center}

For the \emph{only if} part, assume $\mathcal{A}:=A_{\theta_{1}}\rtimes F$ is strongly Morita equivalent to $\mathcal{B}:= A_{\theta_{2}}\rtimes F.$  Let $X$ be an $\mathcal{A}-\mathcal{B}$-imprimitivity bimodule. Define a positive tracial functional $\tau_{X}$ on $\mathcal{B}$ by
$$
\tau_{X}\left(\langle x, y\rangle_{\mathcal{B}}\right):=\Tr^F_{\theta_{1}}(_{\mathcal{A}}\langle y, x\rangle), \quad x, y \in X.
$$
By \cite[Corollary 2.6]{Rie81}, $\Tr^F_{\theta_{1}}$ and $\tau_{X}$ have the same range.  Since $\mathcal{B}$ has a unique trace $\Tr^F_{\theta_{2}}$, $\tau_{X}$ must be a scalar multiple of that trace. 

\end{proof}

Now if $\Tr^F_{\theta_{1}}$ and $\Tr_{\theta_{1}}$ have the same range upto a factor $\frac{1}{|F|}$, and if the same holds for $\Tr^F_{\theta_{2}}$ and $\Tr_{\theta_{2}},$ we have: $\Tr^F_{\theta_{1}}$ and $\lambda \Tr^F_{\theta_{2}}$ have the same range iff $\Tr_{\theta_{1}}$ and $\lambda \Tr_{\theta_{2}}$ have the same range, for some $\lambda>0.$ But the last condition holds iff $A_{\theta_{1}}$ is strongly Morita equivalent to $A_{\theta_{2}}$, using Theorem~\ref{thm:main_Morita}. This observation gives us the following corollaries.

	\begin{corollary}\label{cor:2dim_class}
		Let $\theta_{1}$ and $\theta_{2}$ be irrational numbers and $W$ be one of the matrices $W_{(2)},$ $W_{(3)},$ $W_{(4)},$ $W_{(6)}$ as in Section~\ref{sec:orbifolds_projective} (see also Example~\ref{ex:trace_2_dim}). If $F:=\langle W \rangle,$ then $A_{\theta_{1}}\rtimes F$ is strongly Morita equivalent to $A_{\theta_{2}}\rtimes F$ if and only if $A_{\theta_{1}}$ is strongly Morita equivalent to $A_{\theta_{2}}.$
 	 \end{corollary}
\begin{proof}
	Since the action of $F$ on $\Z^2$ is free outside the origin $0 \in \Z^{2}$, the result follows from the tracial range computation in Example~\ref{ex:trace_2_dim}, and Theorem~\ref{thm:main_Morita}. 
\end{proof}
\noindent The above corollary is not new, see \cite[Theorem 5.3]{BCHL18}.

\begin{corollary}\label{cor:Z_2_main}
		Let $\theta_{1}, \theta_{2} \in \mathcal{T}_n $ be non-degenerate.
		Let $\Z_2$ act on $A_{\theta_{1}}$ and $A_{\theta_{2}}$ by the flip actions.  Then $A_{\theta_{1}}\rtimes \Z_2$ is strongly Morita equivalent to $A_{\theta_{2}}\rtimes \Z_2$ if and only if $A_{\theta_{1}}$ is strongly Morita equivalent to $A_{\theta_{2}}.$
 	 \end{corollary}
 	 \begin{proof}
	Follows similarly as in Corollary~\ref{cor:2dim_class} from the tracial range computation in Example~\ref{ex:trace_flip}, and noting that the action of $\Z_2$ is free outside the origin $0 \in \Z^{n}.$ 
\end{proof}

One can definitely build more examples for which similar results can be stated. For example, a similar result is true for the 4-dimensional example in Example~\ref{ex:diagonal}. 

From Theorem \ref{thm:li_Morita}, $A_{\theta}$ and $A_{g \theta}$ are Morita equivalent, if $g \theta:=\frac{A \theta+B}{C \theta+D}$ is well defined for $g \in \mathrm{SO}(n, n|\Z).$ Note that if $\theta$ is non-degenerate, then $A_{\theta}$ is simple, and hence $A_{g \theta}$ has to be simple so that $g \theta$ is non-degenerate. In the case of two dimensional tori, we have even a stronger result due to Marc Rieffel. Rieffel (in \cite{Rie81}) showed that $A_{\theta}$ and $A_{\theta^{\prime}}$ are Morita equivalent if and only if $\theta$ and $\theta^{\prime}$ are in the same $\GL(2,\mathbb{Z})$ orbit, that is, $\theta^{\prime}=\frac{a \theta+b}{c \theta+d}$ for some matrix $\left(\begin{array}{cc}a & b \\ c & d\end{array}\right)$ in $\GL(2,\mathbb{Z}) .$

It is also known that for a non-degenerate $\theta,$ the fixed point algebra $A_{\theta}^{F}$ is Morita equivalent to the crossed product algebra $A_{\theta}\rtimes F$ (see the proposition in \cite{Ros79}). Hence as a consequence of Corollary~\ref{cor:Z_2_main} we have the following. 
\begin{corollary}\label{cor:Z_2_main_fixed}
		Let $\theta_{1}, \theta_{2}$ and $\Z_2$ be as in Corollary~\ref{cor:Z_2_main}. Then $A_{\theta_{1}}^{\Z_2}$ is strongly Morita equivalent to $A_{\theta_{2}}^{\Z_2}$ if and only if $A_{\theta_{1}}$ is strongly Morita equivalent to $A_{\theta_{2}}.$
 	 \end{corollary}
 	 \noindent A similar statement is true for the 2-dimensional cases (Example~\ref{ex:trace_2_dim}), which follows from Corollary~\ref{cor:2dim_class}.
 
\section*{Acknowledgements}

This research was  supported by DST, Government of India under the \emph{DST-INSPIRE
Faculty Scheme} with Faculty Reg. No. IFA19-MA139.

%----------------------------------------------------------------------------------------
%	BIBLIOGRAPHY
%----------------------------------------------------------------------------------------


\begin{bibdiv}
\begin{biblist}

\bib{BCHL18}{article}{
   author={B\"{o}nicke, Christian},
   author={Chakraborty, Sayan},
   author={He, Zhuofeng},
   author={Liao, Hung-Chang},
   title={Isomorphism and Morita equivalence classes for crossed products of
   irrational rotation algebras by cyclic subgroups of $SL_2(\Z)$},
   journal={J. Funct. Anal.},
   volume={275},
   date={2018},
   number={11},
   pages={3208--3243},
   issn={0022-1236},
   review={\MR{3861734}},
   doi={10.1016/j.jfa.2018.08.008},
}
\bib{BEEK91}{article}{
   author={Bratteli, Ola},
   author={Elliott, George A.},
   author={Evans, David E.},
   author={Kishimoto, Akitaka},
   title={Noncommutative spheres. I},
   journal={Internat. J. Math.},
   volume={2},
   date={1991},
   number={2},
   pages={139--166},
   issn={0129-167X},
   review={\MR{1094701}},
   doi={10.1142/S0129167X91000090},
}
\bib{BM18}{article}{
   author={Benameur, Moulay Tahar},
   author={Mathai, Varghese},
   title={Gap-labelling conjecture with nonzero magnetic field},
   journal={Adv. Math.},
   volume={325},
   date={2018},
   pages={116--164},
   issn={0001-8708},
   review={\MR{3742588}},
   doi={10.1016/j.aim.2017.11.030},
}
\bib{BW07}{article}{
   author={Buck, J.},
   author={Walters, S.},
   title={Connes-Chern characters of hexic and cubic modules},
   journal={J. Operator Theory},
   volume={57},
   date={2007},
   number={1},
   pages={35--65},
   issn={0379-4024},
   review={\MR{2301936}},
}

\bib{Cha20}{article}{
   author={Chakraborty, Sayan},
   title={Some remarks on $\K_0$ of noncommutative tori},
   journal={Math. Scand.},
   volume={126},
   date={2020},
   number={2},
   pages={387--400},
   issn={0025-5521},
   review={\MR{4102570}},
   doi={10.7146/math.scand.a-119699},
}
\bib{CL19}{article}{
   author={Chakraborty, Sayan},
   author={Luef, Franz},
   title={Metaplectic transformations and finite group actions on
   noncommutative tori},
   journal={J. Operator Theory},
   volume={82},
   date={2019},
   number={1},
   pages={147--172},
   issn={0379-4024},
   review={\MR{3979942}},
}
\bib{CU21}{article}{
   author={Chakraborty, Sayan},
      author={Ullisch, Ingo},
title={Morita equivalence of noncommutative orbifolds},
   journal={in preparation},
   date={2021},
  }
\bib{dG11}{book}{
   author={de Gosson, Maurice A.},
   title={Symplectic methods in harmonic analysis and in mathematical
   physics},
   series={Pseudo-Differential Operators. Theory and Applications},
   volume={7},
   publisher={Birkh\"{a}user/Springer Basel AG, Basel},
   date={2011},
   pages={xxiv+337},
   isbn={978-3-7643-9991-7},
   review={\MR{2827662}},
   doi={10.1007/978-3-7643-9992-4},
}

\bib{ELPW10}{article}{
   author={Echterhoff, Siegfried},
   author={L\"{u}ck, Wolfgang},
   author={Phillips, N. Christopher},
   author={Walters, Samuel},
   title={The structure of crossed products of irrational rotation algebras
   by finite subgroups of $SL_2(\Z)$},
   journal={J. Reine Angew. Math.},
   volume={639},
   date={2010},
   pages={173--221},
   issn={0075-4102},
   review={\MR{2608195}},
   doi={10.1515/CRELLE.2010.015},
}
\bib{Ell84}{article}{
   author={Elliott, G. A.},
   title={On the K-theory of the C*-algebra generated by a
   projective representation of a torsion-free discrete abelian group},
   conference={
      title={Operator algebras and group representations, Vol. I},
      address={Neptun},
      date={1980},
   },
   book={
      series={Monogr. Stud. Math.},
      volume={17},
      publisher={Pitman, Boston, MA},
   },
   date={1984},
   pages={157--184},
   review={\MR{731772}},
}
\bib{Ell93}{article}{
   author={Elliott, George A.},
   title={On the classification of $C^*$-algebras of real rank zero},
   journal={J. Reine Angew. Math.},
   volume={443},
   date={1993},
   pages={179--219},
   issn={0075-4102},
   review={\MR{1241132}},
   doi={10.1515/crll.1993.443.179},
}

\bib{He19}{article}{
   author={He, Zhuofeng},
   title={Certain actions of finite abelian groups on higher dimensional
   noncommutative tori},
   journal={M\"{u}nster J. Math.},
   volume={12},
   date={2019},
   number={2},
   pages={473--495},
   issn={1867-5778},
   review={\MR{4030923}},
   doi={10.17879/53149722293},
}
\bib{HG04}{article}{
   author={Higson, Nigel},
   author={Guentner, Erik},
   title={Group C*-algebras and K-theory},
   conference={
      title={Noncommutative geometry},
   },
   book={
      series={Lecture Notes in Math.},
      volume={1831},
      publisher={Springer, Berlin},
   },
   date={2004},
   pages={137--251},
      review={\MR{2058474}},
       doi={10.1007/978-3-540-39702-1\textunderscore3},
  }
\bib{JL15}{article}{
   author={Jeong, Ja A.},
   author={Lee, Jae Hyup},
   title={Finite groups acting on higher dimensional noncommutative tori},
   journal={J. Funct. Anal.},
   volume={268},
   date={2015},
   number={2},
   pages={473--499},
   issn={0022-1236},
   review={\MR{3283161}},
   doi={10.1016/j.jfa.2014.10.010},
}
\bib{LL12}{article}{
   author={Langer, Martin},
   author={L\"{u}ck, Wolfgang},
   title={Topological K-theory of the group C*-algebra of a semi-direct
   product $\Z^n\rtimes\Z/m$ for a free conjugation action},
   journal={J. Topol. Anal.},
   volume={4},
   date={2012},
   number={2},
   pages={121--172},
   issn={1793-5253},
   review={\MR{2949238}},
   doi={10.1142/S1793525312500082},
}
\bib{Li04}{article}{
   author={Li, Hanfeng},
   title={Strong Morita equivalence of higher-dimensional noncommutative
   tori},
   journal={J. Reine Angew. Math.},
   volume={576},
   date={2004},
   pages={167--180},
   issn={0075-4102},
   review={\MR{2099203}},
   doi={10.1515/crll.2004.087},
}
\bib{Lin04}{article}{
   author={Lin, Huaxin},
   title={Classification of simple $C^\ast$-algebras of tracial topological
   rank zero},
   journal={Duke Math. J.},
   volume={125},
   date={2004},
   number={1},
   pages={91--119},
   issn={0012-7094},
   review={\MR{2097358}},
   doi={10.1215/S0012-7094-04-12514-X},
}
\bib{Phi06}{article}{
   author={Phillips, N. Christopher},
   title={Every simple higher dimensional noncommutative torus is an AT algebra},
   journal={arXiv:math/0609783},
 date={2006},
  }
  \bib{Phi11}{article}{
   author={Phillips, N. Christopher},
   title={The tracial Rokhlin property for actions of finite groups on
   $C^\ast$-algebras},
   journal={Amer. J. Math.},
   volume={133},
   date={2011},
   number={3},
   pages={581--636},
   issn={0002-9327},
   review={\MR{2808327}},
   doi={10.1353/ajm.2011.0016},
}


\bib{Rie81}{article}{
   author={Rieffel, Marc A.},
   title={$C^{\ast} $-algebras associated with irrational rotations},
   journal={Pacific J. Math.},
   volume={93},
   date={1981},
   number={2},
   pages={415--429},
   issn={0030-8730},
   review={\MR{623572}},
}

\bib{Rie88}{article}{
   author={Rieffel, Marc A.},
   title={Projective modules over higher-dimensional noncommutative tori},
   journal={Canad. J. Math.},
   volume={40},
   date={1988},
   number={2},
   pages={257--338},
   issn={0008-414X},
   review={\MR{941652}},
   doi={10.4153/CJM-1988-012-9},
}
\bib{RS99}{article}{
   author={Rieffel, Marc A.},
   author={Schwarz, Albert},
   title={Morita equivalence of multidimensional noncommutative tori},
   journal={Internat. J. Math.},
   volume={10},
   date={1999},
   number={2},
   pages={289--299},
   issn={0129-167X},
   review={\MR{1687145}},
   doi={10.1142/S0129167X99000100},
}
\bib{Ros79}{article}{
   author={Rosenberg, Jonathan},
   title={Appendix to: ``Crossed products of UHF algebras by product type actions'' [Duke Math. J. {\textbf 46} (1979), no. 1, 1--23; MR0523598] by O. Bratteli},
   journal={Duke Math. J.},
   volume={46},
   date={1979},
   number={1},
   pages={25--26},
   issn={0012-7094},
   review={\MR{523599}},
}
%\bib{Sla72}{article}{
 %  author={Slawny, Joseph},
  % title={On factor representations and the $C^{\ast} $-algebra of
  % canonical commutation relations},
   %journal={Comm. Math. Phys.},
   %volume={24},
   %date={1972},
   %pages={151--170},
   %issn={0010-3616},
   %review={\MR{293942}},
%}



\bib{Wal95}{article}{
   author={Walters, Samuel G.},
   title={Projective modules over the non-commutative sphere},
   journal={J. London Math. Soc. (2)},
   volume={51},
   date={1995},
   number={3},
   pages={589--602},
   issn={0024-6107},
   review={\MR{1332894}},
   doi={10.1112/jlms/51.3.589},
}
\bib{Wal00}{article}{
   author={Walters, Samuel G.},
   title={Chern characters of Fourier modules},
   journal={Canad. J. Math.},
   volume={52},
   date={2000},
   number={3},
   pages={633--672},
   issn={0008-414X},
   review={\MR{1758235}},
   doi={10.4153/CJM-2000-028-9},
}
  
\end{biblist}
\end{bibdiv}
\end{document}